\documentclass[a4paper, 11pt]{article}

\usepackage[latin1]{inputenc}
\usepackage{graphicx}
\usepackage{amsfonts}
\usepackage{url} 
\usepackage{color}
\usepackage{amsthm} 
\usepackage[lined,boxruled,rightnl]{algorithm2e}
\usepackage{euscript}
\usepackage{multirow}
\usepackage{cite}

\usepackage[margin=3cm]{geometry}

\newtheorem{assumption}{Assumption} 
\newtheorem{proposition}{Proposition} 
 
\newtheorem{lemma}{Lemma} 
\newtheorem{remark}{Remark}

\input alphabet
\input jidef

\newcommand{\rank}{\ensuremath{\operatorname{rank}}}

\newcommand{\Argmind}[2]{\ensuremath{\underset{\substack{{#1}}}%
{\mathrm{Argmin}}\;\;#2 }}

\renewcommand{\leq}{\ensuremath{\leqslant}}
\renewcommand{\geq}{\ensuremath{\geqslant}}
\renewcommand{\le}{\ensuremath{\leqslant}}
\renewcommand{\ge}{\ensuremath{\geqslant}}

\newcommand{\ran}{\ensuremath{\operatorname{ran}}}

\newcommand {\newtext}[1]{\textcolor{black}{#1}}
\newcommand {\revision}[1]{\textcolor{black}{#1}}
\newcommand {\correc}[1]{\textcolor{black}{#1}}

\title{Convergence Rate Analysis of the Majorize-Minimize Subspace Algorithm -- Extended Version}
\author{Emilie Chouzenoux 
and Jean-Christophe Pesquet 
\thanks{E. Chouzenoux (corresponding author)  and J.-C. Pesquet are with the Laboratoire d'Informatique Gaspard 
Monge, UMR CNRS 8049,
Universit\'e Paris-Est, 77454 Marne la Vall\'ee 
Cedex 2, France. E-mail: \texttt{emilie.chouzenoux@univ-paris-est.fr}. This work was supported by the CNRS Imag'In project under grant 2015
OPTIMISME, and by the CNRS Mastodons project under grant 2016 TABASCO.} 
}

\begin{document}

\maketitle

\begin{abstract}
State-of-the-art methods for solving smooth optimization problems are nonlinear conjugate gradient, low memory BFGS,
and Majorize-Minimize (MM) subspace algorithms. The MM subspace algorithm which has been introduced more recently
has shown good practical performance when compared with other methods on various optimization problems arising in signal and image processing. 
However, to the best of our knowledge,  no general result exists concerning the theoretical convergence rate of the MM subspace algorithm.  This paper aims
at deriving such convergence rates both for batch and online versions of the algorithm and, in particular, discusses the influence of the choice of the subspace.
\end{abstract}

{\small \emph{Keywords}: convergence rate, optimization, subspace algorithms, memory gradient methods, descent methods, majorization-minimization, \correc{online optimization, learning}.}

\section{Introduction}

The Majorize-Minimize (MM) subspace algorithm \cite{Chouzenoux2011a} is based on the idea of constructing, at the current iteration, a quadratic majorizing approximation of the cost function of interest~\cite{Hunter04}, and generating the next iterate by minimizing this surrogate function within a subspace spanned by few directions \cite{Elad2007,Conn94,Yuan07}. Note that the MM subspace algorithm can be viewed as a special instance of nonlinear conjugate gradient (NLCG) \cite{Hager06} with closed form formula for the stepsize and conjugacy parameter, or as a particular low memory BFGS (L-BFGS) algorithm \cite{Liu89} with a specific combination of memory directions. The MM subspace algorithm enjoys nice convergence properties \cite{Chouzenoux2012}, and shows good performance in practice, when compared with NLCG, L-BFGS, and also with graph-cut based discrete optimization methods, and proximal algorithms \cite{Chouzenoux2011a,Chouzenoux2011icip,Florescu2014}. It has recently been extended to the online case when only a stochastic approximation of the criterion is employed at each iteration \cite{Chouzenoux15tsp}.
All these works illustrate the fact that the choice of the subspace has a major impact on the practical convergence speed of the algorithm (see, for instance \cite[Section 5]{Chouzenoux2011a}, \cite[Section 5.1]{Chouzenoux2012}). In particular, it seems that the best performance is obtained for the memory gradient subspace \cite{Miele69}, spanned by the current gradient and the previous direction, leading to the so-called MM Memory Gradient (3MG) algorithm. However, only an analysis concerning the convergence rates of half-quadratic algorithms (corresponding to the case when the subspace spans the whole Euclidean space) is available \cite{Allain_M_2006_ieee-tip_On_galc,Nikolova05}.\\
Section \ref{se:algo} describes the general form of the MM subspace algorithm and its main known properties.
In Section \ref{se:rate}, a convergence rate analysis is performed for both batch and online versions of the algorithm for 
\revision{minimizing a wide class of strongly convex cost functions}.

\section{MM subspace algorithm}\label{se:algo}

\subsection{Optimization problem}

In this paper, we will be interested in the minimization of the penalized quadratic cost function:
\revision{
\begin{equation}\label{e:defF}
F\colon  \eR^N \to \eR \colon \hb \mapsto  \frac12 \hb^\top \Rb  \hb - \rb^\top \hb  + \Psi(\hb)\newtext{,}
\end{equation}
where $\rb \in \eR^N$}, $\Rb\in \eR^{N\times N}$ is a symmetric positive definite matrix, and
$\Psi$ is a \correc{lower-bounded} twice-continuously differentiable convex function. In this paper, it will be assumed that $F$ is only accessible through a sequence $(F_n)_{n\ge 1}$ of approximations estimated in an online manner, such that, for every $n\in \eN^*$, 
\revision{
\begin{equation}\label{e:defFn}
 F_n\colon \eR^N \to \eR \colon \hb \mapsto \frac12 \hb^\top \Rb_n \hb - \rb_n^\top \hb  + \Psi(\hb),
\end{equation}
where the vector $\rb_n$ and the symmetric nonnegative definite matrix $\Rb_n$ are approximations of $\rb$ and $\Rb$.  
For simplicity, we will suppose that
\begin{assumption} \label{ass:conver}\
\correc{
\begin{enumerate}
\item  \label{ass:conver1}$(\| \rb_{n} - \rb_{n+1}\|)_{n \ge 1}$ and $(\| \Rb_{n} - \Rb_{n+1}\|)_{n \ge 1}$ are summable sequences,
\item  \label{ass:conver2} $(\rb_n)_{n\ge 1}$, and $(\Rb_n)_{n\ge 1}$ converge to $\rb$ and $\Rb$, respectively.
\end{enumerate} 
}
\end{assumption}
}
It is worth emphasizing that Assumption~\ref{ass:conver} encompasses the batch case when $F_n \equiv F$. Moreover, it should be pointed out that
all the results presented subsequently can be easily  extended to a stochastic framework where 
$\rb_n$ and $\Rb_n$ are consistent statistical estimates of 
$\rb$ and $\Rb$, and convergence arises almost surely.

\subsection{Majorant function}

At each iteration $n\in \eN^*$ of the MM subspace algorithm, the available estimate $F_n$ of $F$ is replaced by a surrogate function $\Theta_n(\cdot,\hb_n)$
based on the current point $\hb_n$ (computed at the previous iteration). \revision{This surrogate function \cite{Jacobson07,Zhang2007,Razaviyayn2016} must be
such that
\begin{equation}
(\forall \hb \in \eR^N) \quad F_n(\hb)-F_n(\hb_n) \le \Theta_n(\hb,\hb_n)-\Theta_n(\hb_n,\hb_n). \label{e:MMmaj}
\end{equation}}
We assume that $\Theta_n(\cdot,\hb_n)$ is a quadratic function of the form
\begin{align}
\label{e:defQc}
(\forall \hb \in \eR^N)\quad
\Theta_n(\hb,\hb_n) =\;
&F_n(\hb_n)+\nabla F_n(\hb_n)^\top (\hb-\hb_n) \nonumber\\
&+ \frac12 (\hb-\hb_n)^\top \Ab_n(\hb_n) (\hb-\hb_n),
\end{align}
where 
$\Ab_n(\hb_n) = \Rb_n + \Bb(\hb_n)$
and $\Bb(\hb_n)\in \eR^{N\times N}$ is some symmetric nonnegative definite matrix (see \cite{Figueiredo2007,Fessler98,RepettiSPL2015,Ning2014,Song15} for examples).
 
\subsection{MM subspace algorithm}

The MM subspace algorithm consists of defining the following sequence of 
vectors $(\hb_n)_{n\ge 1}$:
\begin{equation}\label{e:MMmini}
(\forall n \in \eN^*) \qquad \hb_{n+1} \in \Argmind{\hb \in \ran\Db_n}{ \Theta_n(\hb,\hb_n)},
\end{equation}
where $\hb_1$ is set to an initial value, and $\ran\Db_n$ is the \revision{range of} matrix $\Db_n\in \eR^{N\times M_n}$ with $M_n \ge 1$, constructed in such a way that the steepest descent direction $- \nabla F_n(\hb_n)$ belongs to $\ran\Db_n$. Several choices have been proposed in the literature for matrices $(\Db_n)_{n \in \eN^*}$. On the one hand, if, for every $n\in \eN^*$, $\rank(\Db_n)=N$, Algorithm \eqref{e:MMmini} becomes equivalent to a half-quadratic method with unit stepsize \cite{Allain_M_2006_ieee-tip_On_galc,Idier01,Charbonnier97}. Half-quadratic algorithms are known to be effective optimization methods, but the resolution of the minimization subproblem involved in \eqref{e:MMmini} requires the inversion of matrix $\Ab_n(\hb_n)$ 
which may have a high computational cost. On the other hand, if for every $n\in\eN^*$, $\Db_n$ reduces to $[-\nabla F_n(\hb_n),\hb_n]$, then \eqref{e:MMmini} reads:
for every $n\in \eN^*$
$\hb_{n+1} = u_{n,2} \hb_n-u_{n,1} \nabla F_n(\hb_n)$,
where $(u_{n,1},u_{n,2})\in \mathbb{R}^2$. In the special case when $u_{n,2}=1$, we recover the form of a gradient-like algorithm with step-size $u_{n,1}$ \cite{Labat08,Lange1995}. An intermediate size subspace matrix is obtained by choosing, for every $n>1$,
$\Db_n = 
[-\nabla F_n(\hb_n),\hb_n,\hb_n-\hb_{n-1}]$. 
This particular choice for the subspace yields the 3MG algorithm~\cite{Chouzenoux2012,Chouzenoux15tsp}.

\subsection{Convergence result}

The convergence of the MM subspace Algorithm \eqref{e:MMmini} has been studied in \cite{Chouzenoux2011a,Chouzenoux2012,Chouzenoux15tsp} under various assumptions. We now provide a convergence result which is \revision{a deterministic version of the one in \cite[Section IV]{Chouzenoux15tsp}}. This result requires
the following additional assumption:
\begin{assumption}\label{a:conv}\ 
\begin{enumerate}
\item \label{a:convii} For every $n\in \eN^*$, $\{\nabla F_n(\hb_n),\hb_n\} \subset \ran \Db_n$,
\item \label{a:conviii} There exists a positive definite matrix $\Vb$ such that, for every $n\in \eN^*$, \correc{$ \nabla^2 \Psi(\hb_n) \preceq \Bb(\hb_n) \preceq \Vb$}, where $\nabla^2 \Psi$ denotes the Hessian of $\Psi$, 
\footnote{$\preceq$ and $\prec$ denote
the weak and strict Loewner orders, respectively,}
\item \label{a:conviv} \correc{At least one of the following statements holds:
\begin{enumerate}
	\item \label{a:conviv1} $\rb_n \equiv \rb$ and $\Rb_n \equiv \Rb$,
	\item \label{a:conviv2} $\hb \mapsto  \Bb(\hb) \hb - \nabla \Psi(\hb) $ is a bounded function.
\end{enumerate}
}
\end{enumerate}
\end{assumption}

\correc{
\begin{remark}\label{rmk:lip}
Note that the convexity of $\Psi$ and Assumption \ref{a:conv}\ref{a:conviii} implies that $\Psi$ is Lipschitz differentiable on $\eR^N$, with Lipschitz constant $|||\Vb|||$. Conversely, if $\Psi$ is $\beta$-Lipschitz differentiable with $\beta \in ]0,+\infty[$, Assumption \ref{a:conv}\ref{a:conviii} is satisfied with $\Vb = \Bb(\hb_n) = \beta \Ib_N$~\cite{Bauschkle_2011_book_convex_o}. However, better choices for the curvature matrix are often possible \cite{RepettiSPL2015,Song15}. In particular, Assumption \ref{a:conv}\ref{a:conviv2}, required in the online case, is satisfied for a wide class of functions and majorants \cite{Chouzenoux2011a,Chouzenoux15tsp}.
\end{remark}
}

\begin{proposition}\label{p:convmain}
Assume that Assumptions \ref{ass:conver} and \ref{a:conv} are fulfilled. Then, the following hold:
\begin{enumerate}
\item \label{p:convmain0} $(\|\nabla F_n(\hb_n)\|)_{n\ge 1}$ is square-summable.
\item \label{p:convmainiii} $(\hb_n)_{n\ge 1}$ converges to the unique (global) minimizer $\widehat{\hb}$ of~$F$.
\end{enumerate}
\end{proposition} 

\begin{proof}
See Appendix \ref{ap:convMM}.
\end{proof}

\section{Convergence rate analysis}\label{se:rate}

\subsection{Convergence rate results}

We will first give a technical lemma the proof of which is in the spirit of classical approximation techniques for the study of first-order optimization methods (see \cite[Section 1]{Polak_E_1997_book_Optimization_aca}):
\begin{lemma}\label{le:boundinfFvn}
Suppose that Assumptions \ref{ass:conver} and \ref{a:conv} hold. Let $\epsilon \in ]0,+\infty[$ be such that $\epsilon \Ib_N \prec \Rb $.
Then, there exists $n_\epsilon\in \eN^*$ such that, for every $n\ge n_\epsilon$, 
\revision{$\nabla^2 F_n(\hb_n) \succeq \Rb-\epsilon \Ib_N$}
and
\begin{equation}\label{e:boundinfFvn}
 F_n(\hb_n) - \inf F_n \le \frac12 (1+\epsilon)\big(\nabla F_n(\hb_n)\big)^\top\big(\nabla^2 F_n(\hb_n)\big)^{-1}\nabla F_n(\hb_n).
\end{equation}
\end{lemma}
\begin{proof} See Appendix \ref{ap:boundinfFvn}.
\end{proof}
\medskip
We now state our main result which basically allows us to quantify how fast the proposed iterative approach is able to decrease asymptotically
the cost function: 
\begin{proposition}\label{p:rateconvMM}
Suppose that Assumptions \ref{ass:conver} and \ref{a:conv} hold. Let $\epsilon \in ]0,+\infty[$ be such that $\epsilon \Ib_N \prec \Rb $. 
Then, there exists $n_\epsilon\in \eN^*$ such that, for every $n\ge n_\epsilon$, \revision{$\nabla^2 F_n(\hb_n) \succeq \Rb-\epsilon \Ib_N$} and
\begin{equation}\label{e:boundinfFvnfinal}
F_n(\hb_{n+1})-\inf F_n \le \theta_n
\big(F_n(\hb_n) - \inf F_n\big)
\end{equation}
where $\theta_n = 1-(1+\epsilon)^{-1}\widetilde{\theta}_n$, 
\begin{equation}
 \label{eq:thetantilde} \widetilde{\theta}_n = 
\frac{\big(\nabla F_n(\hb_n)\big)^\top \Cb_n(\hb_n) \nabla F_n(\hb_n)}
{\big(\nabla F_n(\hb_n)\big)^\top\big(\nabla^2F_n(\hb_n)\big)^{-1}\nabla F_n(\hb_n)},
\end{equation}
$\Cb_n(\hb_n) = \Db_n (\Db_n^\top \Ab_n(\hb_n) \Db_n)^\dag  \Db_n^\top$, and $(\cdot)^\dag$ denotes the pseudo-inverse operation. Furthermore, some lower and upper bounds on $\theta_n$ are given by
\begin{align}
&\underline{\theta}_n = 1-(1+\epsilon)^{-1} \underline{\kappa}_n^{-1} \revision{>0},
\label{e:lowertheta}\\
&\revision{\overline{\theta}_n = 1 - (1+\epsilon)^{-1} \overline{\kappa}_n^{-1}\left(1-\Big(\frac{\overline{\sigma}_n-\underline{\sigma}_n}{\overline{\sigma}_n+\underline{\sigma}_n}\Big)^2
\right) <1,}
\label{e:uppertheta}
\end{align}
where
\revision{$\underline{\kappa}_n \ge 1$ (resp. $\overline{\kappa}_n$) is the minimum (resp. maximum) eigenvalue of \linebreak
$\big(\Ab _n(\hb_n)\big)^{\frac12} \big(\nabla^2 F_n(\hb_n)\big)^{-1} \big(\Ab _n(\hb_n)\big)^{\frac12}$, and}
$\underline{\sigma}_n$ (resp. $\overline{\sigma}_n$) is the minimum (resp. maximum) eigenvalue of
\revision{$\nabla^2 F_n(\hb_n)$}.
\end{proposition}
\begin{proof}
See Appendix \ref{ap:rateconvMM}.
\end{proof}

\subsection{Discussion on the choice of the subspace}

Let us make some comments about the above results. First, as enlightened by our proof, at iteration $n\ge n_\epsilon$, the upper value of $\theta_n$ (i.e. the slowest convergence) is obtained in the case of a gradient-like algorithm.
As expected, $\overline{\theta}_n$ has a larger value when the eigenvalues
of the Hessian of $F_n$ are dispersed. Note that, according to \eqref{e:boundHessian},
\begin{equation} \label{eq:bndspectra}
\frac{\overline{\sigma}_n-\underline{\sigma}_n}{\overline{\sigma}_n+\underline{\sigma}_n} 
\le \frac{\overline{\eta}-\underline{\eta}+2\epsilon}{\overline{\eta}+\underline{\eta}},
\end{equation}
where 
$\underline{\eta} > 0$ is the minimum eigenvalue of  $\Rb$ and $\overline{\eta}$ is the maximum eigenvalue of $\Rb  + \Vb$. 
\revision{Since $\big(\big(\Ab _n(\hb_n)\big)^{\frac12} \big(\nabla^2 F_n(\hb_n)\big)^{-1} \big(\Ab _n(\hb_n)\big)^{\frac12}\big)_{n\ge n_\epsilon}$ is bounded,
there exists $\overline{\kappa}_{\rm max} \in [1,+\infty[$ such that $(\forall n \ge n_\epsilon)$ $\overline{\kappa}_n \le \overline{\kappa}_{\rm max}$.}
\revision{All these show} that the decay rate is uniformly strictly lower than 1.

In contrast, when the search subspace is the full space, the lower value of $\theta_n$ (i.e. the fastest convergence)
is obtained. The expression $\underline{\theta}_n$ in \eqref{e:lowertheta}
 shows that the decay is then faster when the quadratic majorant constitutes a tight approximation of function $F_n$ at $\hb_n$. Ideally, if $\Ab_n(\hb_n)$ can be chosen equal to $\nabla^2 F_n(\hb_n)$ and $\Db_n$ is full rank, then $\theta_n = O(\epsilon)$.
Such a behavior similar to Newton's method behavior leads to the best performance one can reasonably expect from the available data at iteration $n$.

Finally, when a mid-size subspace is chosen (as in the 3MG algorithm), an intermediate decay rate is obtained. Provided that $\Db_n$ captures the main eigendirections in $\Ab_n(\hb_n)$, a behavior close to the one previously mentioned can be expected in practice with the potential advantage of a reduced computational complexity per iteration. 

\subsection{Batch case}
The case when $F \equiv F_n$ is of main interest since it is addressed in most of the existing works. Then, Proposition~\ref{p:rateconvMM} and \eqref{eq:bndspectra}
lead to 
\begin{equation}\label{e:boundinfFvfinal}
(\forall n \ge n_\epsilon)\quad 
F(\hb_{n})- \inf F \le \mu \vartheta^n,
\end{equation}
where $\mu =\big(F(\hb_{n_\epsilon}) - \inf F\big)/ \vartheta^{n_\epsilon} $
and the worst-case geometrical decay rate $\vartheta \in ]0,1[$ is given by
\revision{
\begin{equation}
\vartheta = 1- \frac{1}{(1 + \epsilon)\overline{\kappa}_{\rm max}} 
\left(1-\Big(\frac{\overline{\eta}-\underline{\eta}+2\epsilon}{\overline{\eta}+\underline{\eta}}\Big)^2\right).
\end{equation}}
Since $F$ is an $\underline{\eta}$-strongly convex function, the following inequality is satisfied \cite[Definition 10.5]{Bauschkle_2011_book_convex_o}, for every $\alpha \in ]0,1[$, 
\begin{equation}
F\big(\alpha \hb_n+ (1-\alpha) \widehat{\hb}\big)+\frac12 \alpha(1-\alpha) \underline{\eta} \|\hb_n -\widehat{\hb}\|^2  
\le \alpha F(\hb_n)+(1-\alpha) F(\widehat{\hb}),
\end{equation}
\correc{
or, equivalently,
\begin{equation}
\frac12 \alpha(1-\alpha) \underline{\eta} \|\hb_n -\widehat{\hb}\|^2
\le \alpha \big(F(\hb_n)-F(\widehat{\hb})\big)+F(\widehat{\hb})-F\big(\alpha \hb_n+ (1-\alpha) \widehat{\hb}\big).
\end{equation}
Thus,
\begin{equation}
\frac12 (1-\alpha) \underline{\eta} \|\hb_n -\widehat{\hb}\|^2 \le F(\hb_n)-F(\widehat{\hb}).
\end{equation}
Letting $\alpha$ tend to $0$ in the latter inequality implies that}
\begin{equation}
\frac12 \underline{\eta} \|\hb_n -\widehat{\hb}\|^2 \le F(\hb_n)- F(\widehat{\hb}) \le \mu \vartheta^n.
\end{equation}
This shows that the MM subspace algorithm converges linearly with rate $\sqrt{\vartheta}$.

\section{Conclusion}
In this paper, we have established expressions of  the convergence rate of an online
version of the MM subspace algorithm.
These results help in better understanding the good numerical behaviour of this algorithm in signal/image processing applications
and the role played by the subspace choice. Even in the batch case, the provided
linear convergence result appears to be new. In future work, it could be interesting to investigate extensions of these properties to more
general cost functions than \eqref{e:defF}.

\appendix


\section{Proof of Proposition \ref{p:convmain}} \label{ap:convMM}

\subsection{Boundedness of $(\hb_n)_{n \ge 1}$ (online case)}
Assume that Assumption \ref{a:conv}\ref{a:conviv2} holds. For every $n\in \eN^*$, minimizing $\Theta_n(\cdot,\hb_n)$ is equivalent to minimizing the function
\begin{equation}\label{e:deftildeTheta}
(\forall \hb \in \eR^N) \quad 
\widetilde{\Theta}_n(\hb,\hb_n) = \frac12 \hb^\top \Ab_n(\hb_n) \hb-\cb_n(\hb_n)^\top \hb,
\end{equation}
with
\begin{align}
\cb_n(\hb_n) & = \Ab_n(\hb_n)  \hb_n - \nabla F_n(\hb_n) \nonumber \\
& = \rb_n + \Bb(\hb_n) \hb_n -\nabla \Psi(\hb_n)
\end{align}
According to Assumption \ref{a:conv}\ref{a:conviv2}, these exists $\eta \in ]0, +\infty[$ such that
\begin{equation}\label{e:boundc}
(\forall n \ge 1)\qquad \|\cb_n(\hb_n)\| \le \eta,
\end{equation}
In addition, because of Assumption \ref{ass:conver}\ref{ass:conver2}, there exists $\epsilon \in ]0,+\infty[$ and $n_0 \in \eN^*$ such that 
\begin{equation}\label{e:boundA}
(\forall n \ge n_0)\qquad
\Ab_n(\hb_n)  \succeq \Rb-\epsilon \Ib_N \succ \Ob_N,
\end{equation}
Using now the Cauchy-Schwarz inequality, we have
\begin{equation}
(\forall n \ge n_0)
(\forall \hb \in \eR^N) \quad 
\frac12 \hb^\top(\Rb-\epsilon \Ib_N )\hb -\|\hb\| \eta  \le \widetilde{\Theta}_n(\hb,\hb_n).
\end{equation}
Since $\Rb-\epsilon \Ib_N $ is a positive definite matrix, the lower bound corresponds to a coercive function with respect to $\hb$. There thus exists
$\zeta \in ]0,+\infty[$ such that, for every $\hb \in \eR^N$,
\begin{equation}
\|\hb\| > \zeta\quad \Rightarrow \qquad (\forall n \ge n_0)\;\; \widetilde{\Theta}_n(\hb,\hb_n) > 0.
\end{equation}
On the other hand, since $\zerob \in \ran \Db_n $, we have
\begin{equation}
\widetilde{\Theta}_n(\hb_{n+1},\hb_n) \le \widetilde{\Theta}_n(\zerob,\hb_n) = 0.
\end{equation}
The last two inequalities allow us to conclude that
\begin{equation} \label{eq:hbound}
(\forall n \ge n_0)\qquad 
\|\hb_{n+1}\| \le \zeta.
\end{equation}

\subsection{Convergence of $(F_n(\hb_n))_{n \ge 1}$} 
According to Assumption \ref{a:conv}\ref{a:convii}, the proposed algorithm is actually equivalent to
\begin{align}
(\forall n \in \eN^*) \qquad &\hb_{n+1} = \hb_n + \Db_n \widetilde{\ub}_n \label{e:uphvnuvt}\\
& \widetilde{\ub}_n = \argmin_{\widetilde{\ub} \in \eR^{M_n}} \Theta_n(\hb_n + \Db_n \widetilde{\ub},\hb_{n}).
\end{align}
By using \eqref{e:defQc} and cancelling the derivative of the function $\widetilde{\ub} \mapsto \Theta_n(\hb_n + \Db_n \widetilde{\ub},\hb_{n})$,
\begin{equation}\label{e:opttildeuvn}
\Db_n^\top \nabla F_n(\hb_n)  + \Db_n^\top \Ab_n(\hb_n) \Db_n \widetilde{\ub}_n = \zerob.
\end{equation}
Hence,
\begin{align}\label{e:majFnhvnuvt}
 \Theta(\hb_{n+1},\hb_n) & = F_n(\hb_{n})-\frac12 \widetilde{\ub}_n^\top \Db_n^\top \Ab_n(\hb_n) \Db_n \widetilde{\ub}_n\nonumber\\
& = F_n(\hb_{n})-\frac12 (\hb_{n+1}-\hb_n)^\top \Ab_n(\hb_n)  (\hb_{n+1}-\hb_n).
\end{align}
In view of \eqref{e:MMmaj} and \eqref{e:defQc}, this yields
\begin{equation}\label{e:majFnhvnp1hvn}
(\forall n \in \eN^*)\quad
F_n(\hb_{n+1})+ \frac12 (\hb_{n+1}-\hb_n)^\top \Ab_n(\hb_n)  (\hb_{n+1}-\hb_n) \le F_n(\hb_{n}).
\end{equation}
In addition, the following recursive relation holds
\begin{equation}
(\forall \hb \in \eR^N)\quad
 F_{n+1}(\hb)= F_n(\hb) - (\rb_{n+1}-\rb_n)^\top \hb+ \frac12 \hb^\top (\Rb_{n+1}-\Rb_n) \hb.
\end{equation}
It can thus be deduced that
\begin{equation}\label{e:recurFnrand}
F_{n+1}(\hb_{n+1})  + \frac12 (\hb_{n+1}-\hb_n)^\top 
\Ab_n(\hb_n) (\hb_{n+1}-\hb_n) 
 \le \;F_n(\hb_{n})+ \chi_n
\end{equation}
where
\begin{equation}
\chi_n = - (\rb_n-\rb_{n+1})^\top \hb_{n+1} + \frac12 \hb_{n+1}^\top (\Rb_n-\Rb_{n+1})\hb_{n+1}.
\end{equation}
We have
\begin{equation}
|\chi_n| \le \| \rb_n-\rb_{n+1}\| \, \|\hb_{n+1} \|+ \frac12 |||\Rb_n-\Rb_{n+1}||| \, \|\hb_{n+1}\|^2.
\end{equation}
If Assumption \ref{a:conv}\ref{a:conviv2} holds, then, according to \eqref{eq:hbound}, $(\hb_n)_{n\ge 1}$ is bounded, so that Assumption \ref{ass:conver}\ref{ass:conver1} guarantees that
\begin{equation}\label{e:sumchin}
\sum_{n=1}^{+\infty} |\chi_n| <+\infty.
\end{equation}
Otherwise, if Assumption \ref{a:conv}\ref{a:conviv1} holds, then $\chi_n \equiv 0$ and \eqref{e:sumchin} is obviously fulfilled. 
The lower-boundedness property of $\Psi$ entails that, for every $n\in \eN^*$, $F_n$ is lower bounded by
$\inf \Psi > -\infty$. Furthermore, \eqref{e:recurFnrand} leads to
\begin{equation}
 F_{n+1}(\hb_{n+1})-\inf \Psi  + \frac12 (\hb_{n+1}-\hb_n)^\top \Ab_n(\hb_n) (\hb_{n+1}-\hb_n) 
 \le \;F_n(\hb_{n})-\inf \Psi+ |\chi_n|.
\end{equation}
Since, for every $n\in \eN^*$, $F_n(\hb_{n})-\inf \Psi$ and $(\hb_{n+1}-\hb_n)^\top 
\Ab_n(\hb_n) (\hb_{n+1}-\hb_n)$ are nonnegative,  $ \left((\hb_{n+1}-\hb_n)^\top \Ab_n(\hb_n) (\hb_{n+1}-\hb_n)  \right)_{n\ge 1}$ is a summable sequence, and $(F_n(\hb_{n}) )_{n\ge 1}$ is convergent.

\subsection{Convergence of $(\nabla \mathrm{F}_n(\hb_n))_{n \ge 1}$} 
 According to \eqref{e:defQc}, we have, for every $\phi \in \eR$ and $n\in \eN^*$,
\begin{equation}
\Theta_n\big(\hb_n-\phi \nabla F_n(\hb_n),\hb_n\big) = F_n(\hb_n) - \phi \|\nabla F_n(\hb_n)\|^2 
+\frac{\phi^2}{2} \big(\nabla F_n(\hb_n)\big)^\top \Ab_n(\hb_n) \nabla F_n(\hb_n).
\end{equation}
Let
\begin{equation}
 \Phi_n \in \Argmind{\phi \in \eR}{\Theta_n\big(\hb_n-\phi \nabla F_n(\hb_n),\hb_n\big)}. \label{eq:defPhin}
\end{equation}
The following optimality condition holds:
\begin{equation}\label{e:optPhin}
  \big(\nabla F_n(\hb_n)\big)^\top \Ab_n(\hb_n) \nabla F_n(\hb_n)\, \Phi_n = \|\nabla F_n(\hb_n)\|^2.
\end{equation}
As a consequence of Assumption~\ref{a:conv}\ref{a:convii}, $(\forall \phi\in \eR)$ $\hb_n-\phi \nabla F_n(\hb_n) \in \ran \Db_n$.
It then follows from \eqref{e:MMmini} and \eqref{e:optPhin} that 
\begin{equation}\label{e:difFngradFnPhin}
\correc{\Theta_n\big(\hb_{n+1},\hb_n\big) \le \Theta_n\big(\hb_n-\Phi_n \nabla F_n(\hb_n),\hb_n\big) = F_n(\hb_n) -\frac{\Phi_n}{2} \|\nabla F_n(\hb_n)\|^2,}
\end{equation}
which, by using \eqref{e:majFnhvnuvt}, leads to
\begin{equation}\label{e:majgradsq}
\Phi_n \|\nabla F_n(\hb_n)\|^2 \le (\hb_{n+1}-\hb_n)^\top 
\Ab_n(\hb_n) (\hb_{n+1}-\hb_n).
\end{equation}
Let $\epsilon > 0$. Assumption \ref{a:conv}\ref{a:conviii} yields, for every $n\in \eN^*$,
\begin{equation}
\correc{\Ab_n(\hb_n) \preceq  (||| \Rb_n |||+ ||| \Vb ||| )  \Ib_N.}
\end{equation}
Therefore, according to Assumption \ref{ass:conver}\ref{ass:conver2},
\begin{equation}\label{e:boundAA}
(\exists n_0 \in \eN^*)
(\forall n \ge n_0)\quad
\Ob_N\prec \Ab_n(\hb_n)  \preceq  \alpha_\epsilon^{-1} \Ib_N
\end{equation}
where
\begin{equation}\label{e:defalpha}
\correc{\alpha_\epsilon= (|||\Rb |||+ |||\Vb||| +\epsilon)^{-1} > 0.}
\end{equation}
By using now \eqref{e:optPhin}, it can be deduced from \eqref{e:boundAA} that, if $n\ge n_0$
and $\nabla F_n(\hb_n)  \neq \zerob$, then $ \Phi_n  \ge \alpha_\epsilon$.
Then, it follows from \eqref{e:majgradsq} that
\begin{equation}
 \alpha_\epsilon \sum_{n=n_0}^{+\infty}  \|\nabla F_n(\hb_n) \|^2 
 \leq 
 \sum_{n=n_0}^{+\infty} \big(\hb_{n+1} -\hb_n \big)^\top 
\Ab_n(\hb_n) \big(\hb_{n+1} -\hb_n \big).
\end{equation}
By invoking the summability property of $ \left((\hb_{n+1}-\hb_n)^\top \Ab_n(\hb_n) (\hb_{n+1}-\hb_n)  \right)_{n\ge 1}$, we can conclude that $(\|\nabla F_n(\hb_n)\|^2)_{n\ge 1}$ is itself summable.

\subsection{Convergence of $(\hb_{n})_{n \ge 1}$} 
We have shown that $\big((\hb_{n+1}-\hb_n)^\top \Av_n(\hb_n) (\hb_{n+1}-\hb_n)\big)_{n\ge 1}$
converges to 0. In addition, we have seen that \eqref{e:boundA}
holds for a given $\epsilon \in ]0,+\infty[$ and $n_0 \in \eN^*$. This implies
that, for every $n \ge n_0$,
\begin{equation}
|||\Rb-\epsilon \Ib_N |||\, \|\hb_{n+1} -\hb_n \|^2 
 \le \big(\hb_{n+1} -\hb_n \big)^\top \Ab_n(\hb_n) 
\big(\hb_{n+1} -\hb_n \big)
\end{equation}
where $|||\Rb-\epsilon \Ib_N ||| > 0$. Consequently, $(\hb_{n+1}-\hb_n)_{n\ge 1}$ converges to $\zerob$.
In addition, $(\hb_n)_{n\ge 1}$ belongs to a compact set.
Thus, invoking Ostrowski's theorem \cite[Theorem 26.1]{Ostrowki_A_M_1973_book_Solution_eebs} implies that the set of cluster points of $(\hb_n)_{n\ge 1}$ is a nonempty compact connected set. 
By using \eqref{e:defF}-\eqref{e:defFn}, we have
\begin{equation}
(\forall n\in \eN^*)\quad
\nabla F_n(\hb_n)-\nabla F(\hb_n) = (\Rb_n-\Rb)\hb_n - \rb_n+\rb.
\end{equation}
Since $(\hb_n)_{n\ge 1}$ is bounded, it follows from that $\big(\nabla F_n(\hb_n)-\nabla F(\hb_n)\big)_{n\ge 1}$ converges to $\zerob$.
Since $\big(\nabla F_n(\hb_n)\big)_{n\ge 1}$ converges to $\zerob$, this implies that $\big(\nabla F(\hb_n)\big)_{n\ge 1}$ also converges  to $\zerob$. Let $\widehat{\hb}$ be a cluster point of $\big(\hb_n \big)_{n\ge 1}$. There exists a subsequence $\big(\hb_{k_n} \big)_{n\ge 1}$ such that $\hb_{k_n} \to \widehat{\hb}$.
As $F$ is continuously differentiable, we have
\begin{equation} 
\nabla F(\widehat{\hb}) = \lim_{n\to +\infty} \nabla F\big(\hb_{k_n} \big) = \zerob.
\end{equation}
This means that $\widehat{\hb}$ is a critical point of $F$. Since $F$ is a strongly convex function, it possesses a unique critical point $\widehat{\hb}$, which is the global minimizer of $F$~\cite[Prop.11.7]{Bauschkle_2011_book_convex_o}. Since the unique cluster point of $(\hb_n)_{n\ge 1}$ is $\widehat{\hb}$, this shows that $\hb_n \to \widehat{\hb}$.

\section{Proof of Lemma \ref{le:boundinfFvn}} \label{ap:boundinfFvn}
Because $\Rb$ is positive definite, according to Assumption \ref{ass:conver}\ref{ass:conver2}, there
exists $n_0\in \eN^*$ such that, for every $n\ge n_0$, 
\begin{equation}\label{e:bigboundAn}
\boldsymbol{O}_N \prec \Rb -\epsilon \Ib_N  \preceq \Rb_n  \preceq \Rb +\epsilon \Ib_N .
\end{equation}
Let $n\ge n_0$. Then, $F_n$ is a strongly convex continuous function. From standard results, this function possesses a unique global minimizer $\widehat{\hb}_n$. According to Assumption \ref{a:conv}\ref{a:conviii}, and \eqref{e:bigboundAn}, $\nabla^2 F_n$ is such that
\begin{align}\label{e:boundHessian}
(\forall \hb \in \eR^N)\quad \Ob_N &\prec \Rb-\epsilon \Ib_N  \nonumber\\
&\preceq \Rb_n +\nabla^{(2)} \Psi(\hb)  = \nabla^2 F_n(\hb)\nonumber\\
& \preceq \Rb +\epsilon \Ib_N   +  \Vb.
\end{align}
By using now the second-order Taylor formula with integral remainder, we get
\begin{equation}\label{e:Taylor2F}
F_n(\widehat{\hb}_n) = F_n(\hb_n)+\big(\nabla F_n(\hb_n)\big)^\top \big(\widehat{\hb}_n-\hb_n\big) 
+ \frac12 \big(\widehat{\hb}_n-\hb_n\big)^\top \Hb_n^{(2)}(\hb_n) \big(\widehat{\hb}_n-\hb_n\big),
\end{equation}
where
\begin{align}\label{e:Taylor1F}
\nabla F_n(\hb_n)&=\nabla F_n(\widehat{\hb}_n)+\Hb_n^{(1)}(\hb_n) (\hb_n-\widehat{\hb}_n)\nonumber\\
&= \Hb_n^{(1)}(\hb_n) (\hb_n-\widehat{\hb}_n)
\end{align}
and, for every $\hb \in \eR^N$,
\begin{align}
\Hb_n^{(1)}(\hb) &=
\int_0^1 \nabla^2 F_n\big(\widehat{\hb}_n+t(\hb-\widehat{\hb}_n)\big) dt\nonumber\\ 
&= \Rb_n +\int_0^1 \nabla^2 \Psi\big(\widehat{\hb}_n+t(\hb-\widehat{\hb}_n)\big) dt\\
\Hb_n^{(2)}(\hb) &=
2 \int_0^1 (1-t)\nabla^2 F_n\big(\widehat{\hb}_n+t(\hb-\widehat{\hb}_n)\big) dt\nonumber\\
&= \Rb_n +2\int_0^1 (1-t)\nabla^2 \Psi\big(\widehat{\hb}_n+t(\hb-\widehat{\hb}_n)\big) dt.
\end{align}
Because of the lower bound in \eqref{e:boundHessian}, 
\begin{equation}\label{e:lowboundHn1}
(\forall \hb \in \eR^N)\quad \Ob_N \prec \Rb-\epsilon \Ib_N  \preceq \Hb_n^{(1)}(\hb)
\end{equation}
and $\Hb_n^{(1)}(\hb)$ is thus invertible. Therefore,
combining \eqref{e:Taylor2F} and \eqref{e:Taylor1F} yields
\begin{multline}\label{e:FvnFvnhvn}
F_n(\widehat{\hb}_n) = F_n(\hb_n)
- \big(\nabla F_n(\hb_n)\big)^\top \big(\Hb_n^{(1)}(\hb_n)\big)^{-1} \nabla F_n(\hb_n) \\
 + \frac12 \big(\nabla F_n(\hb_n)\big)^\top \big(\Hb_n^{(1)}(\hb_n)\big)^{-1} \Hb_n^{(2)}(\hb_n) \big(\Hb_n^{(1)}(\hb_n)\big)^{-1} \nabla F_n(\hb_n).
\end{multline}
According to Assumption~\ref{a:conv}\ref{a:conviii}, for every $t\in [0,1]$, 
\begin{equation}
||| \nabla^2 \Psi\big(\widehat{\hb}_n+t(\hb_n-\widehat{\hb}_n)\big)||| \le  |||\Vb |||,
\end{equation} 
\revision{where  $|||\cdot|||$ denotes the matrix spectral norm.}
As Proposition~\ref{p:convmain}\ref{p:convmainiii} guarantees that $(\hb_n)_{n\ge 1}$ converges to the unique minimizer
$\widehat{\hb}$ of $F$, it follows from Proposition \ref{p:convmain}\ref{p:convmain0}, \eqref{e:Taylor1F}, and \eqref{e:lowboundHn1} that $(\widehat{\hb}_n)_{n\ge 1}$ also converges to $\widehat{\hb}$.
By using the continuity of $\nabla^2 \Psi$, $\big(\nabla^2 \Psi\big(\widehat{\hb}_n+t(\hb_n-\widehat{\hb}_n)\big)\big)_{n\ge 1}$ converges to
$\nabla^2 \Psi(\widehat{\hb})$ and, by invoking the dominated convergence theorem, it can be deduced that
\begin{equation}
\int_0^1 \nabla^2 \Psi\big(\widehat{\hb}_n+t(\hb_n-\widehat{\hb}_n)\big) dt \to \nabla^2 \Psi(\widehat{\hb}) .
\end{equation}
Since $(\Rb_n)_{n\ge 1}$ converges to $\Rb$, this allows us to conclude that $\big(\Hb_n^{(1)}(\hb_n)\big)_{n\ge 1}$
converges to $\nabla^2 F(\widehat{\hb})$. Proceeding similarly, it can be proved
that $\big(\Hb_n^{(2)}(\hb_n)\big)_{n\ge 1}$ also converges to $\nabla^2 F(\widehat{\hb})$.
This entails that
\begin{equation}
 \big(\Hb_n^{(1)}(\hb_n)\big)^{-1}
- \frac12 \big(\Hb_n^{(1)}(\hb_n)\big)^{-1} \Hb_n^{(2)}(\hb_n) \big(\Hb_n^{(1)}(\hb_n)\big)^{-1} 
 \to \frac12 \big(\nabla^2 F(\widehat{\hb})\big)^{-1}.
\end{equation}
Besides,  since $\big(\nabla^2 F_n(\hb_n)\big)_{n\ge 1} = \big(\Rb_n + \nabla^2 \Psi(\hb_n)\big)_{n\ge 1}$
converges to $\nabla^2 F(\widehat{\hb})$, there exists $n_\epsilon \ge n_0$ such that, for every $n\ge n_\epsilon$,
\begin{align}
\big(\Hb_n^{(1)}(\hb_n)\big)^{-1}- \frac12 \big(\Hb_n^{(1)}(\hb_n)\big)^{-1} \Hb_n^{(2)}(\hb_n) \big(\Hb_n^{(1)}(\hb_n)\big)^{-1} 
& -\frac12 \big(\nabla^2 F_n(\hb_n)\big)^{-1}\\
&\preceq \frac12 \epsilon (\Rb+\epsilon \Ib_N+ \Vb)^{-1}\nonumber\\
&\preceq \frac12 \epsilon\big(\nabla^2 F_n(\hb_n)\big)^{-1},
\end{align}
where the last inequality follows from \eqref{e:boundHessian}.
This implies that
\begin{equation}
 \big(\Hb_n^{(1)}(\hb_n)\big)^{-1}
- \frac12 \big(\Hb_n^{(1)}(\hb_n)\big)^{-1} \Hb_n^{(2)}(\hb_n) \big(\Hb_n^{(1)}(\hb_n)\big)^{-1} 
 \preceq \frac12 (1+\epsilon)  \big(\nabla^2 F_n(\hb_n)\big)^{-1}.
\end{equation}
By coming back to \eqref{e:FvnFvnhvn}, we deduce that, for every $n\ge n_\epsilon$,
\eqref{e:boundinfFvn} holds.

\section{Proof of Proposition \ref{p:rateconvMM}} \label{ap:rateconvMM}

Let $n \in \eN^*$. If $\nabla F_n(\hb_n)$ is zero, then $\hb_n$ is a global minimizer of $F_n$ and, \correc{according to \eqref{e:MMmaj}-\eqref{e:MMmini}, $F(\hb_{n+1}) \le \Theta_n(\hb_{n+1},\hb_n) - \Theta_n(\hb_{n},\hb_n) + F(\hb_n) \le F(\hb_n)$ so that $\hb_{n+1}$ is also a global minimizer of $F_n$, and \eqref{e:boundinfFvnfinal} is obviously satisfied}. So, without loss of generality, it will be assumed
in the rest of the proof that $\nabla F_n(\hb_n)$ is nonzero.
Because of Assumption \ref{a:conv}\ref{a:conviii} and \eqref{e:bigboundAn}, there
exists $n_0\in \eN^*$ such that,
for every 
$n\ge n_0$, 
\begin{equation}\label{e:bigboundAn2}
\boldsymbol{O}_N \prec \Rb -\epsilon \Ib_N \preceq \Rb_n  \preceq \Ab_n(\hb_n).
\end{equation}
\correc{
Using \eqref{e:majFnhvnp1hvn} and the definition of $\Cb_n(\hb_n)$, 
\begin{align}\label{e:devFvnFvn1}
F_n(\hb_{n+1}) \le & \;F_n(\hb_{n})-\frac12 (\hb_{n+1}-\hb_n)^\top \Ab_n(\hb_n)  (\hb_{n+1}-\hb_n)\nonumber\\
 = &\;F_n(\hb_{n})-\revision{\frac12 \left(\nabla F_n(\hb_n)\right)^\top} \Cb_n(\hb_n) \nabla F_n(\hb_n).
\end{align}
Combining \eqref{e:bigboundAn2}, \eqref{e:devFvnFvn1} and \eqref{e:difFngradFnPhin} yields}
\begin{equation}\label{e:lowerboundrate}
\frac{\|\nabla F_n(\hb_n)\|^4}{\big(\nabla F_n(\hb_n)\big)^\top \Ab_n(\hb_n) \nabla F_n(\hb_n)} 
\le \big(\nabla F_n(\hb_n)\big)^\top \Cb_n(\hb_n) \nabla F_n(\hb_n).
\end{equation}
In turn, we have
\begin{equation}\label{e:thetanminhvtilde}
 \Theta_n\big(\widetilde{\hb}_{n},\hb_n\big) \le \Theta_n\big(\hb_{n+1},\hb_n\big),
\end{equation}
where $\widetilde{\hb}_{n}$ is a global minimizer of $\Theta_n(\cdot,\hb_n\big)$. If 
$n \ge n_0$, then \eqref{e:bigboundAn2} shows that $\Ab_n(\hb_n)$ is invertible, and
\begin{equation}
\widetilde{\hb}_n = \hb_n- \big(\Ab_n(\hb_n)\big)^{-1} \nabla F_n(\hb_n)
\end{equation}
which, by using \eqref{e:devFvnFvn1} and \eqref{e:thetanminhvtilde}, yields
\begin{equation}\label{e:upperboundrate}
\big(\nabla F_n(\hb_n)\big)^\top \Cb_n(\hb_n) \nabla F_n(\hb_n) 
\le \big(\nabla F_n(\hb_n)\big)^\top \big(\Ab_n(\hb_n)\big)^{-1} \nabla F_n(\hb_n).
\end{equation}
It can be noticed that the lower bound in \eqref{e:lowerboundrate} is obtained when $\Db_n = \nabla F_n(\hb_n)$,
while the upper bound in \eqref{e:upperboundrate} is attained when $M_n = N$ and $\Db_n$ is full rank.

Let us now apply Lemma \ref{le:boundinfFvn}. According to this lemma, there exists $n_\epsilon\ge n_0$ such that, for every $n\ge n_\epsilon$,
\eqref{e:boundinfFvn} holds with $\nabla^2F_n(\hb_n) \succ \Ob_N$.
Let us assume that $n\ge n_\epsilon$.
By combining 
\eqref{e:boundinfFvn} and  \eqref{e:devFvnFvn1},  
we obtain
\begin{align}\label{e:boundinfFvnfinalpre}
&F_n(\hb_{n})-F_n(\hb_{n+1})
 \ge \frac{\widetilde{\theta}_n}{1+\epsilon}
\big(F_n(\hb_n) - \inf F_n\big)\nonumber\\
\Leftrightarrow \;\;&F_n(\hb_{n+1})-\inf F_n \le \Big(1-\frac{\widetilde{\theta}_n}{1+\epsilon}\Big)
\big(F_n(\hb_n) - \inf F_n\big),
\end{align}
which itself is equivalent to \eqref{e:boundinfFvnfinal}. The following lower bound is then be deduced from \eqref{e:lowerboundrate}:
\begin{equation}
\widetilde{\theta}_n \ge 
\frac{\|\nabla F_n(\hb_n)\|^4}{\beta_n \big(\nabla F_n(\hb_n)\big)^\top \Ab_n(\hb_n) \nabla F_n(\hb_n)},
\end{equation}
by setting $\beta_n = \big(\nabla F_n(\hb_n)\big)^\top\big(\nabla^2F_n(\hb_n)\big)^{-1}\nabla F_n(\hb_n)$.
%
\revision{
Hence, we have
\begin{align}\label{e:apreskantorotoro}
\widetilde{\theta}_n &\ge 
\frac{\|\nabla F_n(\hb_n)\|^4}{\beta_n \beta_n'} \frac{\big(\nabla F_n(\hb_n)\big)^\top \nabla^2 F_n(\hb_n) \nabla F_n(\hb_n)}{\nabla F_n(\hb_n)\big)^\top \Ab_n(\hb_n) \nabla F_n(\hb_n)},
\nonumber\\
&\ge \frac{\|\nabla F_n(\hb_n)\|^4}{\beta_n \beta_n'} \Bigg(\sup_{\substack{\gb \in \eR^N\\\gb \neq \zerob}}\; 
\frac{\gb^\top \Ab_n(\hb_n) \gb}{\gb^\top \nabla^2 F_n(\hb_n) \gb}\Bigg)^{-1},
\end{align}
where $\beta_n' =  \big(\nabla F_n(\hb_n)\big)^\top \nabla^2 F_n(\hb_n) \nabla F_n(\hb_n)$.
The sup term in \eqref{e:apreskantorotoro} corresponds to the generalized Rayleigh quotient of 
$\Ab_n(\hb_n)$ and $\nabla^2 F_n(\hb_n)$, which is equal to $\overline{\kappa}_n$.
By invoking now Kantorovich inequality \cite[Section 1.3.2]{Polak_E_1997_book_Optimization_aca}, we get
\begin{equation}\label{e:Kanto}
\widetilde{\theta}_n \ge \frac{4 \underline{\sigma}_n \overline{\sigma}_n}{\overline{\kappa}_n(\underline{\sigma}_n+\overline{\sigma}_n)^2},
\end{equation}
}
which leads to
\begin{equation}\label{e:kantorobis}
1-\frac{\widetilde{\theta}_n}{1+\epsilon} \le \overline{\theta}_n <1
\end{equation}
since \revision{$\overline{\sigma}_n \geq \underline{\sigma}_n > 0$}.
An upper bound on $\widetilde{\theta}_n$ is derived from \eqref{e:upperboundrate} and \eqref{eq:thetantilde}:
\revision{
\begin{align}\label{e:apreskantoro}
\widetilde{\theta}_n &\le 
\frac{\big(\nabla F_n(\hb_n)\big)^\top \big(\Ab_n(\hb_n)\big)^{-1} \nabla F_n(\hb_n)}{\big(\nabla F_n(\hb_n)\big)^\top\big(\nabla^2 F_n(\hb_n)\big)^{-1}\nabla F_n(\hb_n)}\nonumber\\ 
&\le \sup_{\substack{\gb \in \eR^N\\\gb \neq \zerob}}\; \frac{\gb^\top \big(\Ab _n(\hb_n)\big)^{-1} \gb}{\gb^\top\big(\nabla^2 F_n(\hb_n)\big)^{-1}\gb}.
\end{align}}
The sup term in \eqref{e:apreskantoro} 
is equal to \revision{$\underline{\kappa}_n^{-1}$}.
Altogether \eqref{e:boundinfFvnfinalpre}, \eqref{e:kantorobis}, and \eqref{e:apreskantoro} yield \eqref{e:boundinfFvnfinal}-\eqref{e:uppertheta},
by setting $\theta_n = 1-(1+\epsilon)^{-1}\widetilde{\theta}_n$.
\revision{In view of Assumption \ref{a:conv}\ref{a:conviii} and the equality in \eqref{e:boundHessian}, 
the Hessian of $F_n$ is such that
\begin{equation}\label{e:boundHessian2}
(\forall \hb \in \eR^N)\quad \nabla^2 F_n(\hb)  \preceq \Ab_n(\hb),
\end{equation}
and therefore $\underline{\kappa}_n \ge 1$.}

\bibliographystyle{IEEEbib}
\bibliography{stoch}

\begin{thebibliography}{10}

\bibitem{Chouzenoux2011a}
E.~Chouzenoux, J.~Idier, and S.~Moussaoui,
\newblock ``A majorize-minimize subspace strategy for subspace optimization
  applied to image restoration,''
\newblock {\em IEEE Trans. Image Process.}, vol. 20, no. 18, pp. 1517--1528,
  Jun. 2011.

\bibitem{Hunter04}
D.~R. Hunter and K.~Lange,
\newblock ``A tutorial on {MM} algorithms,''
\newblock {\em {A}mer. {S}tat.}, vol. 58, no. 1, pp. 30--37, Feb. 2004.

\bibitem{Elad2007}
M.~Elad, B.~Matalon, and M.~Zibulevsky,
\newblock ``Coordinate and subspace optimization methods for linear least
  squares with non-quadratic regularization,''
\newblock {\em Appl. Comput. Harmon. Anal.}, vol. 23, pp. 346--367, Nov. 2007.

\bibitem{Conn94}
A.~R. Conn, N.~Gould, A.~Sartenaer, and Ph.~L. Toint,
\newblock ``On iterated-subspace minimization methods for nonlinear
  optimization,''
\newblock Tech. {R}ep. 94-069, Rutherford Appleton Laboratory, Oxfordshire, UK,
  May 1994,
\newblock ftp://130.246.8.32/pub/reports/cgstRAL94069.ps.Z.

\bibitem{Yuan07}
Y.~Yuan,
\newblock ``Subspace techniques for nonlinear optimization,''
\newblock in {\em Some Topics in Industrial and Applied Mathematics},
  R.~Jeltsh, T.-T. Li, and H~I. Sloan, Eds., vol.~8, pp. 206--218. Series on
  Concrete and Applicable Mathematics, 2007.

\bibitem{Hager06}
W.~W. Hager and H.~Zhang,
\newblock ``A survey of nonlinear conjugate gradient methods,''
\newblock {\em Pac. J. Optim.}, vol. 2, no. 1, pp. 35--58, Jan. 2006.

\bibitem{Liu89}
D.~C. Liu and J.~Nocedal,
\newblock ``On the limited memory {BFGS} method for large scale optimization,''
\newblock {\em Math. Program.}, vol. 45, no. 3, pp. 503--528, Aug. 1989.

\bibitem{Chouzenoux2012}
E.~Chouzenoux, A.~Jezierska, J.-C. Pesquet, and H.~Talbot,
\newblock ``A majorize-minimize subspace approach for $\ell_2$-$\ell_0$ image
  regularization,''
\newblock {\em SIAM J. Imag. Sci.}, vol. 6, no. 1, pp. 563--591, 2013.

\bibitem{Chouzenoux2011icip}
E.~Chouzenoux, J.-C. Pesquet, H.~Talbot, and A.~Jezierska,
\newblock ``A memory gradient algorithm for $\ell_2$-$\ell_0$ regularization
  with applications to image restoration,''
\newblock in {\em 18th IEEE Int. Conf. Image Process. (ICIP 2011)}, Brussels,
  Belgium, 11-14 Sep. 2011, pp. 2717--2720.

\bibitem{Florescu2014}
A.~Florescu, E.~Chouzenoux, J.-C. Pesquet, P.~Ciuciu, and S.~Ciochina,
\newblock ``A majorize-minimize memory gradient method for complex-valued
  inverse problem,''
\newblock {\em Signal Process.}, vol. 103, pp. 285--295, Oct. 2014,
\newblock Special issue on Image Restoration and Enhancement: Recent Advances
  and Applications.

\bibitem{Chouzenoux15tsp}
E.~Chouzenoux and J.-C. Pesquet,
\newblock ``A stochastic majorize-minimize subspace algorithm for online
  penalized least squares estimation,''
\newblock Tech. {R}ep., 2015,
\newblock http://arxiv.org/abs/1512.08722.

\bibitem{Miele69}
A.~Miele and J.~W. Cantrell,
\newblock ``Study on a memory gradient method for the minimization of
  functions,''
\newblock {\em J. Optim. Theory Appl.}, vol. 3, no. 6, pp. 459--470, Nov. 1969.

\bibitem{Allain_M_2006_ieee-tip_On_galc}
M.~Allain, J.~Idier, and Y.~Goussard,
\newblock ``On global and local convergence of half-quadratic algorithms,''
\newblock {\em IEEE Trans. Image Process.}, vol. 15, no. 5, pp. 1130--1142, May
  2006.

\bibitem{Nikolova05}
M.~Nikolova and M.~Ng,
\newblock ``Analysis of half-quadratic minimization methods for signal and
  image recovery,''
\newblock {\em SIAM J. Sci. Comput.}, vol. 27, no. 3, pp. 937--966, 2005.

\bibitem{Jacobson07}
M.~W. Jacobson and J.~A. Fessler,
\newblock ``An expanded theoretical treatment of iteration-dependent
  {M}ajorize-{M}inimize algorithms,''
\newblock {\em {IEEE} {T}rans. {I}mage {P}rocess.}, vol. 16, no. 10, pp.
  2411--2422, Oct. 2007.

\bibitem{Zhang2007}
Z.~Zhang, J.~T. Kwok, and D.-Y. Yeung,
\newblock ``Surrogate maximization/minimization algorithms and extensions,''
\newblock {\em Mach. Learn.}, vol. 69, pp. 1--33, Oct. 2007.

\bibitem{Razaviyayn2016}
M.~Hong, M.~Razaviyayn, Z.~Q. Luo, and J.~S. Pang,
\newblock ``A unified algorithmic framework for block-structured optimization
  involving big data: With applications in machine learning and signal
  processing,''
\newblock {\em IEEE Signal Process. Mag.}, vol. 33, no. 1, pp. 57--77, Jan.
  2016.

\bibitem{Figueiredo2007}
M.~Figueiredo, J.~Bioucas-Dias, and R.~Nowak,
\newblock ``Majorization-minimization algorithms for wavelet-based image
  restoration,''
\newblock {\em IEEE Trans. Image Process.}, vol. 16, no. 12, pp. 2980--2991,
  Dec. 2007.

\bibitem{Fessler98}
J.A. Fessler and H.~Erdogan,
\newblock ``A paraboloidal surrogates algorithm for convergent
  penalized-likelihood emission image reconstruction,''
\newblock Toronto, Canada, 8-14 Nov. 1998, vol.~2, pp. 1132--1135.

\bibitem{RepettiSPL2015}
A.~Repetti, M.~Q. Pham, L.~Duval, E.~Chouzenoux, and J.-C. Pesquet,
\newblock ``Euclid in a taxicab: Sparse blind deconvolution with smoothed l1/l2
  regularization,''
\newblock {\em IEEE Signal Process. Letters}, vol. 22, no. 5, pp. 539--543, May
  2015.

\bibitem{Ning2014}
X.~Ning, I.~W. Selesnick, and L.~Duval,
\newblock ``Chromatogram baseline estimation and denoising using sparsity
  (beads),''
\newblock {\em Chemometr. Intell. Lab. Syst.}, vol. 139, pp. 156--167, 2014.

\bibitem{Song15}
J.~Song, P.~Babu, and D.~P. Palomar,
\newblock ``Sparse generalized eigenvalue problem via smooth optimization,''
\newblock {\em IEEE Trans. Signal Process.}, vol. 63, no. 7, pp. 1627--1642,
  Apr. 2015.

\bibitem{Idier01}
J.~Idier,
\newblock ``Convex half-quadratic criteria and interacting auxiliary variables
  for image restoration,''
\newblock {\em IEEE Trans. Image Process.}, vol. 10, no. 7, pp. 1001--1009,
  Jul. 2001.

\bibitem{Charbonnier97}
P.~Charbonnier, L.~Blanc-F\'eraud, G.~Aubert, and M.~Barlaud,
\newblock ``Deterministic edge-preserving regularization in computed imaging,''
\newblock {\em IEEE Trans. Image Process.}, vol. 6, no. 2, pp. 298--311, Feb.
  1997.

\bibitem{Labat08}
C.~Labat and J.~Idier,
\newblock ``Convergence of conjugate gradient methods with a closed-form
  stepsize formula,''
\newblock {\em J. Optim. Theory Appl.}, vol. 136, no. 1, pp. 43--60, Jan. 2008.

\bibitem{Lange1995}
K.~Lange,
\newblock ``A gradient algorithm locally equivalent to the {EM} algorithm,''
\newblock {\em J. R. Stat. Soc. Series B Stat. Methodol.}, vol. 57, no. 2, pp.
  425--437, 1995.

\bibitem{Bauschkle_2011_book_convex_o}
H.~H. Bauschke and P.~L. Combettes,
\newblock {\em Convex Analysis and Monotone Operator Theory in {H}ilbert
  Spaces},
\newblock Springer, New York, 2011.

\bibitem{Polak_E_1997_book_Optimization_aca}
E.~Polak,
\newblock {\em Optimization. Algorithms and Consistent Approximations},
\newblock Springer-Verlag, New York, 1997.

\bibitem{Ostrowki_A_M_1973_book_Solution_eebs}
A.~M. Ostrowski,
\newblock {\em Solution of Equations in Euclidean and Banach Spaces},
\newblock Academic Press, London, 1973.

\end{thebibliography}

\end{document}